\documentclass[10pt, reqno]{amsart}
\usepackage{amsmath, amsthm, amscd, amsfonts, amssymb, graphicx, color}
\usepackage[bookmarksnumbered, colorlinks, plainpages]{hyperref}
\hypersetup{colorlinks=true,linkcolor=red, anchorcolor=green, citecolor=cyan, urlcolor=red, filecolor=magenta, pdftoolbar=true}
\usepackage{mathrsfs}
\newtheorem{theorem}{Theorem}[section]
\newtheorem{lemma}[theorem]{Lemma}

\newtheorem{corollary}[theorem]{Corollary}
\theoremstyle{definition}
\newtheorem{definition}[theorem]{Definition}
\newtheorem{example}[theorem]{Example}

\theoremstyle{remark}
\newtheorem{remark}[theorem]{Remark}
\numberwithin{equation}{section}

\begin{document}
\setcounter{page}{1}

\title[Ordered braid groups]{Remark on ordered braid groups}

\author[Nikolaev]
{Igor V. Nikolaev$^1$}

\address{$^{1}$ Department of Mathematics and Computer Science, St.~John's University, 8000 Utopia Parkway,  
New York,  NY 11439, United States.}
\email{\textcolor[rgb]{0.00,0.00,0.84}{igor.v.nikolaev@gmail.com}}

%\dedicatory{In memory of Ola Bratteli}
%\dedicatory{There is no conflict of interest,  funding, or other authors.  All data are available as part of this manuscript. }



%\date{Received:  August 14, 2015; Revised: yyyyyy; Accepted: zzzzzz.}

\begin{abstract}
We recover the Dehornoy order on the braid group $B_{2g+n}$ 
from the tracial state on a cluster $C^*$-algebra $\mathbb{A}(S_{g,n})$
associated to the  surface $S_{g,n}$ of genus $g$ with $n$ boundary components. 
It is proved that the space of left-ordering of the 
fundamental group $\pi_1(S_{g,n})$ is a totally
disconnected dense subspace of the  projective Teichm\"uller space $\mathbb{P}T_{g,n}\cong \mathbf{S}^{6g-7+2n}$. 
In particular, each left-ordering of $\pi_1(S_{g,n})$ defines the orbit of a Riemann surface $S_{g,n}$
under the geodesic flow on the space $T_{g,n}$. 
\end{abstract}

\maketitle

%**************************************************************************
\section{Introduction}
%***************************************************************************
Let $S_{g,n}$ be an orientable surface of genus $g\ge 0$ with $n\ge 1$ boundary 
components and let  $2g-2+n>0$. The fundamental group $\pi_1(S_{g,n})$ is a free group of rank $2g+n-1$. 
Such a group  encodes the homotopy type of $S_{g,n}$. 
The  surface $S_{g,n}$ can be endowed with 
a  complex structure  so that it becomes a Riemann surface with $n$ cusps. 
Such an assignment encodes geometry of  $S_{g,n}$ and  there exists infinitely many 
distinct  Riemann surfaces assigned to the same $S_{g,n}$.
The latter make a continuum $T_{g,n}\cong \mathbf{R}^{6g-6+2n}$
called the  Teichm\"uller space  of $S_{g,n}$. 
One can ask what an extra structure of the  group $\pi_1(S_{g,n})$
is responsible for geometry of the surface  $S_{g,n}$? 
The aim of our note is to answer this question, see Corollary \ref{cor1.2}.

Recall that the left (right, resp.) strict total order $\prec$ on the group $G$ is a
relation between the elements of $G$,  such that 
$x\prec y$ implies $zx\prec zy$ ($xz\prec yz$, resp.) for all $x,y,z\in G$.
If both the left and right order exists, the group is said to be orderable. 
The orderable groups exist and the same group $G$ may have  a continuum 
of different orderings.  In particular, the  group $\pi_1(S_{g,n})$
is  orderable   [Rolfsen 2014] \cite[Proposition 2.9]{Rol1}. 
The left-orderings of $G$ have a natural topology and the corresponding space 
 will be  denoted  by $LO(G)$ [Sikora 2004] \cite{Sik1}.

 To formalize our results, we need the following definitions.
The  cluster algebra  $\mathcal{A}(\mathbf{x}, B)$
is a subring  of the field  of  rational functions in 
 variables  $\mathbf{x}=(x_1,\dots, x_m)$
defined by a skew-symmetric matrix  $B=(b_{ij})\in M_n(\mathbf{Z})$.
A new cluster $\mathbf{x}'=(x_1,\dots,x_k',\dots,  x_m)$ and a new
skew-symmetric matrix $B'=(b_{ij}')$ is obtained from 
$(\mathbf{x}, B)$ by the   exchange relations (\ref{eq2.1}), see  [Williams 2014]  \cite{Wil1}
for the motivation and examples.
The seed $(\mathbf{x}', B')$ is said to be a  mutation of $(\mathbf{x}, B)$ in the direction $k$,
where $1\le k\le m$. The  algebra  $\mathcal{A}(\mathbf{x}, B)$ is  generated by the 
cluster  variables $\{x_i\}_{i=1}^{\infty}$
obtained from the initial seed $(\mathbf{x}, B)$ by the iteration of mutations  in all possible
directions $k$.  
The  Laurent phenomenon
 says  that  $\mathcal{A}(\mathbf{x}, B)\subset \mathbf{Z}[\mathbf{x}^{\pm 1}]$,
where  $\mathbf{Z}[\mathbf{x}^{\pm 1}]$ is the ring of  the Laurent polynomials in  variables $\mathbf{x}=(x_1,\dots,x_m)$.
Therefore each  generator $x_i$  of  the algebra $\mathcal{A}(\mathbf{x}, B)$  can be 
written as a  Laurent polynomial in $n$ variables with the   integer coefficients.

The  algebra  $\mathcal{A}(\mathbf{x}, B)$  has the structure of an additive abelian
semigroup consisting of the Laurent polynomials with positive coefficients. 
In other words,  the $\mathcal{A}(\mathbf{x}, B)$ is a dimension group  \cite[Definition 3.5.2]{N}.
The cluster $C^*$-algebra  $\mathbb{A}(\mathbf{x}, B)$  is   an  AF-algebra,  
such that $K_0(\mathbb{A}(\mathbf{x}, B))\cong  \mathcal{A}(\mathbf{x}, B)$
\cite[Section 4.4]{N}.
 We denote by   $\mathcal{A}(\mathbf{x},  S_{g,n})$ the cluster algebra 
 coming from  a triangulation of the surface $S_{g,n}$ 
  [Williams 2014]  \cite[Section 3.3]{Wil1}.
  The corresponding cluster $C^*$-algebra will be denoted by $\mathbb{A}(S_{g,n})$.

Let $B_m$ be the braid group given by the generators $\sigma_1,\dots\sigma_{m-1}$
 subject to the relations $\sigma_i\sigma_j=\sigma_j\sigma_i$ when $|i-j|>1$
 and $\sigma_i\sigma_j\sigma_i=\sigma_j\sigma_i\sigma_j$ when $|i-j|=1$. 
 The group $B_m$ is left-orderable \cite[Theorem I.1]{DDRW}. 
 Recall that for $n\in\{1, 2\}$ there exists a faithful representation \cite[Theorem 4.4.1]{N}:
 %*****************************************************************************
 \begin{equation}\label{eq1.1}
\left\{\rho: B_{2g+n}\to \mathbb{A}(S_{g,n}) ~|~ \sigma_i\mapsto e_i+1,  ~\hbox{where 
$e_i$ is projection}\right\}.
 \end{equation}
 %*********************************************************************
 
 \bigskip
 The representation (\ref{eq1.1}) defines an ordering  on  the group   $B_{2g+n}$
 as follows. (We refer the reader to Sections 2.2 and 2.3 for the details.) The self-adjoint element $x\in  \mathbb{A}(S_{g,n})$
 is called positive if all its eigenvalues are non-negative real numbers.  The set of all positive elements is a cone $\mathbb{A}^+(S_{g,n})$
 in the $C^*$-algebra $\mathbb{A}(S_{g,n})$.
 Recall that  $K_0(\mathbb{A}(S_{g,n}))\cong  \mathcal{A}(\mathbf{x}, S_{g,n})$
 and let $\mathcal{A}^+(\mathbf{x}, S_{g,n})$ be the additive semigroup 
 consisting of the Laurent polynomials   $\mathbf{Z}[\mathbf{x}^{\pm 1}]$
 with positive coefficients. 
The $K_0$-group  preserves positivity of the respective
semigroups, i.e. $K_0(\mathbb{A}^+(S_{g,n}))\cong  \mathcal{A}^+(\mathbf{x}, S_{g,n})$. 
We define  the semigroup $B_{2g+n}^+$  as the composition of mappings $\rho$ and $K_0$ shown on   
the commutative diagram  in Figure 1.  Our main results can be formulated as follows.
 %***********************************************************************
\begin{theorem}\label{thm1.1}
The ordering  of the braid group $B_{2g+n}$ defined by 
 $B_{2g+n}^+$ is:

\medskip
(i) left-invariant;

\smallskip
(ii) isomorphic to the Dehornoy order.
\end{theorem}
%**************************************************************************
%*******************************************************************
\begin{figure}
%*******************************************************************
\begin{picture}(300,90)(-70,20)
\put(80,70){\vector(0,-1){30}}
%\put(130,70){\vector(0,-1){35}}
\put(30,70){\vector(1,-1){30}}
\put(50,83){\vector(1,0){20}}
\put(65,25){$ \mathcal{A}^+(\mathbf{x}, S_{g,n})$}
\put(17,80){$B_{2g+n}^+$}
\put(75,80){$\mathbb{A}^+(S_{g,n})$}
\put(55,90){$\rho$}
\put(85,55){$K_0$}
%\put(65,55){$p'$}
\end{picture}
%***********************************************************
\caption{}
\end{figure}
%*******************************************************************

%***********************************************************************
\begin{remark}\label{rmk1.2}
If  $g=0, ~n=2$  ($g=n=1$, resp.) then  
 the cluster algebra $\mathcal{A}^+(\mathbf{x}, S_{0,2})$ ($\mathcal{A}^+(\mathbf{x}, S_{1,1})$, resp.)
 consists of the Jones (HOMFLY, resp.) polynomials \cite[Section 4.4.6]{N}. 
In this case the map $B_{2g+n}^+\to \mathcal{A}^+(\mathbf{x}, S_{g,n})$ 
can be described explicitly. 
Namely, such a map acts by the formula
 %*****************************************************************************
 %\begin{equation}\label{eq1.2}
$b\mapsto V_{\hat b}(t)$  ~($b\mapsto W_{\hat b}(s,t)$, resp.),
%\end{equation}
 %*********************************************************************
where $V_{\hat b}(t)$ ($W_{\hat b}(s,t)$, resp.) are the Jones 
(HOMFLY, resp.) polynomials of the closure $\hat b$ of the braid $b$.
This phenomenon was reported by  [Franks \& Williams 1987] \cite[Theorem 2.2 (1)]{FraWil1} 
and [Ito 2020] \cite{Ito1}, respectively.  
A generalization to the multivariable Laurent polynomials
is given by \cite[Theorem 4.4.1]{N}.  
\end{remark}
%**************************************************************************

Let $\{T^t: T_{g,n}\to T_{g,n}~|~t\in\mathbf{R}\}$ be  the 
Teichm\"uller geodesic flow on $T_{g,n}$ \cite[Section 4]{Nik1}. 
By a  projective Teichm\"uller space  $\mathbb{P}T_{g,n}\cong \mathbf{S}^{6g-7+2n}$
we understand the space of orbits of the flow $T^t$, where $\mathbf{S}^{6g-7+2n}$
is the sphere of dimension $6g-7+2n$. 
Denote by $\mathscr{K}$  a totally disconnected dense subset of $\mathbf{S}^{6g-7+2n}$ 
consisting of the vectors with the rationally independent components.  
%***********************************************************************
\begin{corollary}\label{cor1.2}
$LO (\pi_1(S_{g,n}))\cong \mathscr{K}\subset\mathbf{S}^{6g-7+2n}\cong\mathbb{P} T_{g,n}$.
In particular, each left-ordering of $\pi_1(S_{g,n})$ defines 
the orbit $\{T^t(S_{g,n}) ~|~t\in\mathbf{R}\}$ of a Riemann
surface $S_{g,n}$ under the geodesic flow $T^t$.
\end{corollary}
%**************************************************************************
The article is organized as follows. 
The preliminary facts can be found in Section 2.
Theorem \ref{thm1.1} and corollary \ref{cor1.2}  are proved in Section 3. 
In Section 4,  we apply corollary \ref{cor1.2} to evaluate the space $LO(F_m)$
of the free group of rank $m$.

%**************************************************************************
\section{Preliminaries}
%***************************************************************************
We briefly review the ordered braid groups, positivity for the $C^*$-algebras
and cluster $C^*$-algebras. 
We refer the reader to  [Dehornoy, Dynnikov, Rolfsen and Wiest 2002] \cite{DDRW},
[Williams 2014]  \cite{Wil1} and \cite[Section 4.4]{Nik1} for the details.

%**************************************************************************
\subsection{Positive braids}
%***************************************************************************
%**************************************************************************
\subsubsection{Ordered groups}
%***************************************************************************
A strict ordering of a set $X$ is a binary relation $\prec$ which is transitive,
i.e. $x\prec y$ and $y\prec z$ implies $x\prec z$, and 
anti-reflexive, i.e. $x\prec x$ cannot hold for all $x,y,z\in X$.  
A strict ordering of $X$ is called linear (or total) 
of for all $x,y\in X$ one and only one of the three relations
is possible: $x\prec y$, $y\prec x$ or $x=y$. 
%****************************************************************
\begin{definition}
The group $G$ is left-orderable (right-orderable, resp.) 
if $G$ admits a total order which is invariant by the left
(right, resp.) multiplication, i.e. for all $x,y,z\in G$
if $x\prec y$ then $zx\prec zy$ ($xz\prec yz$, resp.). 
\end{definition}
%*****************************************************************
%****************************************************************
\begin{remark}\label{rmk2.2}
Every subgroup $H$ of the left-orderable group $G$ is left-orderable.
If $h: G\to H$ is a group homomorphism, then $H$ is left-orderable if
$G$ is left-orderable.  Indeed, the order $h(x)\prec h(y)$ of $H$ 
induced by $x\prec y$ is left-invariant,  since $h(z)h(x)\prec h(z)h(y)$
for all $h(z)\in H$.  The converse is false in general.  
\end{remark}
%*****************************************************************
%****************************************************************
\begin{remark}\label{rmk2.3}
The group $G$ is left-orderable if and only if $G$ 
is right-orderable. Indeed, given a left-order $\prec$ on $G$
one can define a new order $\prec^*$ by letting 
$x\prec^* y$ whenever $y^{-1}\prec x^{-1}$.
It is verified directly that   $\prec^*$ is a right-invariant
order. The converse is proved similarly. 
\end{remark}
%*****************************************************************
An element $x$ of the left-orderable group $G$ is called positive
(negative, resp.) if $Id\prec x$ ($x\prec Id$, resp.). 
The set of all positive (negative, resp.) elements is a semigroup
$G^+$ ($G^-$, resp.) and $G=G^+\cup G^-\cup \{Id\}$ 
is a disjoint union,  provided the order is total. 
The semigroup $G^+$ ($G^-$, resp.) will be called a positive
(negative, resp.) cone in $G$.

%**************************************************************************
\subsubsection{Dehornoy order}
%***************************************************************************
The $m$-strand braid group $B_m$ is given by the  presentation
%**********************************************************************************
\begin{equation}
\langle\sigma_1,\dots,\sigma_{m-1}~|~ 
\sigma_i\sigma_j=\sigma_j\sigma_i ~\hbox{for} ~|i-j|\ge 2, 
~\sigma_i\sigma_j\sigma_i=\sigma_j\sigma_i\sigma_j  ~\hbox{for}  ~|i-j|=1
\rangle. 
\end{equation}
%**********************************************************************************
%****************************************************************
\begin{definition}
An element $x\in B_m$ is called $\sigma_i$-positive if it contains
$\sigma_i$ but does not contain $\sigma_i^{-1}$ or 
$\sigma_j^{\pm 1}$ for $j<i$.
\end{definition}
%*****************************************************************
%****************************************************************
\begin{remark}\label{rmk2.5}
It is verified directly that the set $B_m^+$ of all $\sigma_i$-positive 
elements  of $B_m$  is a semigroup closed under the multiplication operation. 
\end{remark}
%*****************************************************************
%****************************************************************
\begin{theorem}\label{thm2.6}
{\bf (Dehornoy \cite{DDRW})}
The semigroup $B_m^+$ defines a left-invariant linear order
 of the braid  group $B_m$. 
\end{theorem}
%*****************************************************************

%**************************************************************************
\subsection{Positivity for $C^*$-algebras}
%***************************************************************************
%**************************************************************************
\subsubsection{$C^*$-algebras}
%***************************************************************************
The  $C^*$-algebra is an algebra $\mathcal{A}$ over $\mathbf{C}$ with a norm
$a\mapsto ||a||$ and an involution $a\mapsto a^*$ such that
it is complete with respect to the norm and $||ab||\le ||a||~ ||b||$
and $||a^*a||=||a||^2$ for all $a,b\in \mathcal{A}$.

The spectrum $Spec~a$ of an element $a\in\mathcal{A}$ is the set
of complex numbers $\lambda$ such that $a-\lambda \mathbf{1}$ is not
invertible.  An element $a\in\mathcal{A}$ is self-adjoint  if $a^*=a$. 
The spectrum of a self-adjoint element is a subset of the real line. 
If element $a$ is self-adjoint and $Spec~a\subseteq [0,\infty)$,
then $a$ is called   positive.  
The set of all positive elements of the $C^*$-algebra $\mathcal{A}$ is denoted by 
$\mathcal{A}^+$. 
The set $\mathcal{A}^+$. is a closed cone in $\mathcal{A}$, i.e. 
if $x,y\in \mathcal{A}^+$ then $x+y\in \mathcal{A}^+$ and
$\lambda x\in \mathcal{A}^+$ for any real $\lambda>0$. 

A linear functional $\phi$ on $\mathcal{A}$  is positive if $\phi(x)\ge 0$
whenever $x\in\mathcal{A}^+$. A positive linear functional of norm $1$
is called a state on  $\mathcal{A}$. The set of all states $S(\mathcal{A})$ 
is the state space of  $\mathcal{A}$. A tracial state $\tau\in S(\mathcal{A})$ 
is a state satisfying $\tau(yx)=\tau(xy)$ for all $x, y\in \mathcal{A}$.

%**************************************************************************
\subsubsection{$K$-theory}
%***************************************************************************
 The algebraic direct limit of the $C^*$-algebras 
$M_n(\mathcal{A})$ under the embeddings $a\mapsto ~\mathbf{diag} (a,0)$
will be denoted  by $M_{\infty}(\mathcal{A})$. 
Two projections $p,q\in M_{\infty}(\mathcal{A})$ are equivalent, if there exists 
an element $v\in M_{\infty}(\mathcal{A})$,  such that $p=v^*v$ and $q=vv^*$. 
Let $[p]$ be the equivalence class of projection $p$.   
We write $V(\mathcal{A}):=\{[p] ~:~ p=p^*=p^2\in M_{\infty}(\mathcal{A})\}$ 
for all equivalence classes of 
projections in the $C^*$-algebra $M_{\infty}(\mathcal{A})$.
The set $V(\mathcal{A})$ has the natural structure of an abelian 
semi-group with the addition operation defined by the formula 
$[p]+[q]:=\mathbf{diag}(p,q)=[p'\oplus q']$, where $p'\sim p, ~q'\sim q$ 
and $p'\perp q'$.  The identity of the semi-group $V(\mathcal{A})$ 
is given by $[0]$, where $0$ is the zero projection. 
By the $K_0$-group $K_0(\mathcal{A})$ of the unital $C^*$-algebra $\mathcal{A}$
one understands the Grothendieck group of the abelian semi-group
$V(\mathcal{A})$, i.e. a completion of $V(\mathcal{A})$ by the formal elements
$[p]-[q]$.  The image of $V(\mathcal{A})$ in  $K_0(\mathcal{A})$ 
is a positive cone $K_0^+(\mathcal{A})$ defining  the order structure $\le$  on the  
abelian group  $K_0(\mathcal{A})$. The pair   $\left(K_0(\mathcal{A}),  K_0^+(\mathcal{A})\right)$
is known as a dimension group of the $C^*$-algebra $\mathcal{A}$. 
%****************************************************************
\begin{remark}\label{rmk2.7}
The $K_0$-group preserves the positive cone $\mathcal{A}^+$.
Namely, a tracial sate $\tau$ on $\mathcal{A}$ extends to such on 
$M_{\infty}(\mathcal{A})$ so that the abelian semigroup $V(\mathcal{A})$
corresponds to the positive cone $\mathcal{A}^+$.  
\end{remark}
%*****************************************************************

%**************************************************************************
\subsection{Cluster $C^*$-algebras}
%***************************************************************************
%**************************************************************************
\subsubsection{Cluster algebras}
%***************************************************************************
The cluster algebra  of rank $n$ 
is a subring  $\mathcal{A}(\mathbf{x}, B)$  of the field  of  rational functions in $n$ variables
depending  on  variables  $\mathbf{x}=(x_1,\dots, x_n)$
and a skew-symmetric matrix  $B=(b_{ij})\in M_n(\mathbf{Z})$.
The pair  $(\mathbf{x}, B)$ is called a  seed.
A new cluster $\mathbf{x}'=(x_1,\dots,x_k',\dots,  x_n)$ and a new
skew-symmetric matrix $B'=(b_{ij}')$ is obtained from 
$(\mathbf{x}, B)$ by the   exchange relations [Williams 2014]  \cite[Definition 2.22]{Wil1}:
%*********************************************************************************************
\begin{eqnarray}\label{eq2.1}
x_kx_k'  &=& \prod_{i=1}^n  x_i^{\max(b_{ik}, 0)} + \prod_{i=1}^n  x_i^{\max(-b_{ik}, 0)},\cr 
b_{ij}' &=& 
\begin{cases}
-b_{ij}  & \mbox{if}   ~i=k  ~\mbox{or}  ~j=k\cr
b_{ij}+{|b_{ik}|b_{kj}+b_{ik}|b_{kj}|\over 2}  & \mbox{otherwise.}
\end{cases}
\end{eqnarray}
%******************************************************************************************* 
The seed $(\mathbf{x}', B')$ is said to be a  mutation of $(\mathbf{x}, B)$ in direction $k$.
where $1\le k\le n$.  The  algebra  $\mathcal{A}(\mathbf{x}, B)$ is  generated by the 
cluster  variables $\{x_i\}_{i=1}^{\infty}$
obtained from the initial seed $(\mathbf{x}, B)$ by the iteration of mutations  in all possible
directions $k$.   The  Laurent phenomenon
 says  that  $\mathcal{A}(\mathbf{x}, B)\subset \mathbf{Z}[\mathbf{x}^{\pm 1}]$,
where  $\mathbf{Z}[\mathbf{x}^{\pm 1}]$ is the ring of  the Laurent polynomials in  variables $\mathbf{x}=(x_1,\dots,x_n)$
 [Williams 2014]  \cite[Theorem 2.27]{Wil1}.
In particular, each  generator $x_i$  of  the algebra $\mathcal{A}(\mathbf{x}, B)$  can be 
written as a  Laurent polynomial in $n$ variables with the   integer coefficients.

 %**************************************************************************
\subsubsection{Cluster $C^*$-algebras}
%***************************************************************************
 The cluster algebra  $\mathcal{A}(\mathbf{x}, B)$  has the structure of an additive abelian
semigroup consisting of the Laurent polynomials with positive coefficients. 
In other words,  the $\mathcal{A}(\mathbf{x}, B)$ is a dimension group  \cite[Definition 3.5.2]{Nik1}.
We define the cluster $C^*$-algebra  $\mathbb{A}(\mathbf{x}, B)$  as   an  AF-algebra,  
such that $K_0(\mathbb{A}(\mathbf{x}, B))\cong  \mathcal{A}(\mathbf{x}, B)$
\cite[Section 4.4]{Nik1}.

%**************************************************************************
\subsubsection{Surface cluster algebras and braid groups}
%***************************************************************************
 Denote by $S_{g,n}$  the Riemann surface   of genus $g\ge 0$  with  $n\ge 0$ cusps.
 Let   $\mathcal{A}(\mathbf{x},  S_{g,n})$ be the cluster algebra 
 coming from  a triangulation of the surface $S_{g,n}$   [Williams 2014]  \cite[Section 3.3]{Wil1}. 
 We shall denote by  $\mathbb{A}(S_{g,n})$  the corresponding cluster $C^*$-algebra. 
 Since $\mathbb{A}(S_{g,n})$ is an AF-algebra, it is characterized by its K-theory and
 thus can be generated by a series of projections $\{e_i\}_{i=1}^{\infty}$.  
 We let $e_i$ be such  a projection in the cluster $C^*$-algebra  $\mathbb{A}(S_{g,n})$. 
%****************************************************************
\begin{theorem}\label{thm2.8}
{\bf (\cite[Theorem 4.4.1]{N})}
The formula $\sigma_i\mapsto e_i+1$ defines a faithful representation 
%*******************************************************************
\begin{equation}\label{eq2.3}
\rho: \begin{cases}
B_{2g+1}\to \mathbb{A}(S_{g,1}) &\cr
B_{2g+2}\to \mathbb{A}(S_{g,2}). &
\end{cases}
\end{equation}
%*********************************************************************
If  $b\in B_{2g+1}$  ($b\in B_{2g+2}$, resp.) is a braid,   there exists 
a Laurent polynomial  $[\rho(b)]\in   K_0(\mathbb{A}(S_{g,1}))$   
($[\rho(b)] \in K_0(\mathbb{A}(S_{g,2}))$, resp.)
with the integer coefficients    depending on  $2g$  ($2g+1$, resp.) 
variables,  such that  $[\rho(b)]$  is a topological invariant of the closure of 
  $b$.  
\end{theorem}
%*****************************************************************

%**************************************************************************
\section{Proof}
%***************************************************************************
%**************************************************************************
\subsection{Proof of theorem \ref{thm1.1}}
%***************************************************************************
%**************************************************************************
\subsubsection{Part I}
%***************************************************************************
 We split the proof of item (i) of theorem \ref{thm1.1} in a series of lemmas.
%********************************************************************
\begin{lemma}\label{lm3.1}
Let $A_{2g+n}$ be the norm-closure in the strong operator topology of the algebra 
$\mathbf{C}\langle\rho (\sigma_1),\dots,\rho (\sigma_{2g+n-1})\rangle$,
where $\rho$ is given by the formula  (\ref{eq1.1}). Then:

\medskip
(i) $A_{2g+n}$ is a $C^*$-subalgebra of the AF-algebra $\mathbb{A}(S_{g,n})$;

\smallskip
(ii) $\left(K_0(A_{2g+n}),  ~K_0^+(A_{2g+n})\right)\cong \left(K_0(\mathbb{A}(S_{g,n})),  ~K_0^+(\mathbb{A}(S_{g,n}))\right)$
is an isomorphism of the dimension groups.
\end{lemma}
%*********************************************************************
\begin{proof}
(i) Since $\rho(\sigma_i)=e_i+1$, we conclude that 
$\rho^*(\sigma_i)=(e_i+1)^*=e_i^*+1^*=e_i+1=\rho(\sigma_i)$,
i.e. $\rho(\sigma_i)$ is a self-adjoint element. It is verified directly,
that the braid relations are invariant of the $\ast$-conjugacy. Therefore,
the polynomial algebra
$\mathbf{C}\langle\rho (\sigma_1),\dots,\rho (\sigma_{2g+n-1})\rangle$
is an $\ast$-algebra. Clearly, the norm-closure in the strong operator 
topology of the $\ast$-algebra is a $C^*$-algebra. Thus 
 $A_{2g+n}$ is a $C^*$-subalgebra of the AF-algebra $\mathbb{A}(S_{g,n})$. 
 Item (i) is proved.

\bigskip
(ii) In view of item (i), one gets an injective $*$-homomorphism  $A_{2g+n}\to \mathbb{A}(S_{g,n})$.
The latter gives rise to an injective homomorphism between the abelian groups 
$K_0(A_{2g+n})\to K_0(\mathbb{A}(S_{g,n}))$ by the main property of the functor $K_0$. 
In other words, we obtain an inclusion of the dimension  groups: 
%*********************************************************************
\begin{equation}\label{eq3.1}
\left(K_0(A_{2g+n}),  ~K_0^+(A_{2g+n})\right)\subseteq \left(K_0(\mathbb{A}(S_{g,n})),  ~K_0^+(\mathbb{A}(S_{g,n}))\right).
\end{equation}
%******************************************************************* 

\medskip
Let us show that (\ref{eq3.1})  is an isomorphism of the dimension groups. 
Indeed, assume to the contrary that $\left(K_0(A_{2g+n}),  ~K_0^+(A_{2g+n})\right)\subset
\left(K_0(\mathbb{A}(S_{g,n})),  ~K_0^+(\mathbb{A}(S_{g,n}))\right)$ is a strict inclusion. 
Recall  \cite[Theorem 2]{Nik1} that there exists an order-ideal: 
%*******************************************************************************
\begin{equation}
I_{\theta}\subset  \left(K_0(\mathbb{A}(S_{g,n})),  ~K_0^+(\mathbb{A}(S_{g,n}))\right):=G,
\end{equation}
%*******************************************************************************
such that $G/I_{\theta}$ is a simple dimension group. 
Then $I_{\theta} ~\cap ~\left(K_0(A_{2g+n}),  ~K_0^+(A_{2g+n})\right):=I_{\theta}'$
is an order-ideal of the dimension group  $\left(K_0(A_{2g+n}),  ~K_0^+(A_{2g+n})\right):=G'$
such that $G'/I_{\theta}'$ is a non-trivial ideal of $G/I_{\theta}$. 
The latter contradicts simplicity of the dimension group $G/I_{\theta}$. 
Thus (\ref{eq3.1}) must be an isomorphism between the corresponding dimension groups. 
Item (ii) follows. 

\bigskip
Lemma \ref{lm3.1} is proved.
\end{proof}
%*********************************************************************

%********************************************************************
\begin{corollary}\label{cor3.2}
$K_0(A_{2g+n})\cong \mathcal{A}(\mathbf{x}, S_{g,n}).$  
\end{corollary}
%*********************************************************************
\begin{proof}
By definition $K_0(\mathbb{A}(S_{g,n}))\cong  \mathcal{A}(\mathbf{x}, S_{g,n})$,  where 
 $\mathcal{A}(\mathbf{x}, S_{g,n})$ is the additive group of the surface cluster algebra. 
 Corollary \ref{cor3.2}  follows from item (ii) of lemma \ref{lm3.1}. 
 \end{proof}
%*********************************************************************

\bigskip
%********************************************************************
\begin{remark}\label{rmk3.3}
By  the construction,   the $C^*$-algebra $A_{2g+n}$
is  isomorphic to a reduced group $C^*$-algebra $C_r^*(B_{2g+n})$  
of the braid group $B_{2g+n}$. 
Indeed, since $\rho(\sigma_i)=e_i+1$ one finds its inverse  
$\rho(\sigma_i^{-1})=\frac{1}{2}(2-e_i)$ and therefore $\rho$
is a unitary representation. 
Moreover, it can be shown that   $\rho$  is weakly contained in the
regular representation of $B_{2g+n}$.
Thus corollary \ref{cor3.2} yields an explicit  formula: 
%*********************************************************************
\begin{equation}\label{eq3.2}
K_0(C_r^*(B_{2g+n}))\cong \mathcal{A}(\mathbf{x}, S_{g,n})\subset \mathbf{Z}[\mathbf{x}^{\pm 1}]. 
\end{equation}
%******************************************************************* 
In other words, the $K_0$-group of the reduced group $C^*$-algebra
of the braid group consists of the Laurent polynomials with the integer coefficients.
To the best of our knowledge, the expression (\ref{eq3.2})  is new. 
An independent study of the K-theory  for the pure braids was 
undertaken in the recent paper  [Azzali, Browne, Gomez Aparicio, Ruth \& Wang 2022] \cite{AzzBroGomRutWan1}
and  is based on the  Baum-Connes isomorphism. 
\end{remark}
%*********************************************************************

\bigskip
%********************************************************************
\begin{lemma}\label{lm3.4}
The positive cone $B_{2g+n}^+$ defines a left-invariant ordering of the 
braid group $B_{2g+n}$. 
\end{lemma}
%*********************************************************************
\begin{proof}
(i) In view of remark \ref{rmk2.7}, the positive cone $B_{2g+n}^+$ 
consists of the words which correspond to the elements of $\mathcal{A}^+(\mathbf{x}, S_{g,n})$,
i.e. the Laurent polynomials with positive coefficients. 
Let us show that the  ordering  of  the  group $B_{2g+n}$ induced by $B_{2g+n}^+$
is left-invariant.

\bigskip
(ii) Roughly speaking, the property of invariance of the ordering under left-multiplication
 follows from the  property $\tau(vu)=\tau(uv)$ of the tracial state $\tau$ holding
for all $u,v\in C_r^*(B_{2g+n})$.  Namely,  it is immediate that the left-multiplication 
invariance is equivalent to:
%*********************************************************************
\begin{equation}\label{eq3.3}
\tau(zx)=\tau(x), \quad \forall z\in B_{2g+n}. 
\end{equation}
%******************************************************************* 

\bigskip
(iii) Therefore it is necessary and sufficient to show that the following system of equations
in variables $u$ and $v$
is solvable for any $x,z\in B_{2g+n}$: 
%*******************************************************************
\begin{equation}\label{eq3.4} 
\begin{cases} 
vu=zx&\cr
uv=x.&
\end{cases}
\end{equation}
%*********************************************************************

\bigskip
(iv)  It is easy to see that (\ref{eq3.4}) is equivalent to: 
%*******************************************************************
\begin{equation}\label{eq3.5} 
\begin{cases} 
v=u^{-1}x&\cr
zx=u^{-1}xu.&
\end{cases}
\end{equation}
%*********************************************************************

\bigskip
(v)  To solve (\ref{eq3.5}), recall \cite[Proposition 1.1.5]{DDRW} that for every braid $\beta\in B_{2g+n}$ 
there exist positive braids $\beta_1,\beta_2\in B_{2g+n}^+$, 
such that: 
 %*******************************************************************
\begin{equation}\label{eq3.6} 
\beta=\beta_1^{-1}\beta_2. 
\end{equation}
%*********************************************************************
 
\bigskip
(vi)  We let  $\beta=x$ and $\beta_1=z$.  Then (\ref{eq3.6}) takes the form:
  %*******************************************************************
\begin{equation}\label{eq3.7} 
zx=\beta_2. 
\end{equation}
%*********************************************************************

\bigskip
(vii) 
Recall that  each $x\in B_{2g+n}$ is equivalent to a positive element $\beta_2\in B_{2g+n}^+$,
see e.g.  \cite[p. 9]{DDRW}.  In other words, $\beta_2=u^{-1}xu$ for some $u\in B_{2g+n}$.

\bigskip
(viii) Therefore there exists $u\in B_{2g+n}$  such that 
equation  (\ref{eq3.7}) takes  the form: 
  %*******************************************************************
\begin{equation}\label{eq3.8} 
zx=u^{-1}xu
\end{equation}
%*********************************************************************
  
\bigskip
(ix) It remains to notice that $u$ in  (\ref{eq3.8}) is a root of the 
second equation (\ref{eq3.5}).    The second root $v=u^{-1}x$
is defined by the values of $u$ and $x$. 
Thus equations (\ref{eq3.4}) are solvable for the variables $u$ and $v$.
As explained, the latter is equivalent to the order-invariance of the left multiplication
in the braid group $B_{2g+n}$. 

\bigskip
Lemma \ref{lm3.4} is proved. 
 \end{proof}
%*********************************************************************

%**************************************************************************
\subsubsection{Part II}
%***************************************************************************
Our proof  of item (ii) of theorem \ref{thm1.1} is based on an observation, that the closure of positive braids in the Dehornoy order
gives rise to  positive Laurent polynomilas,  see remark \ref{rmk1.2}. 
This fact was proved by [Franks \& Williams 1987] \cite[Theorem 2.2 (1)]{FraWil1} 
and [Ito 2020] \cite{Ito1} for the Jones and HOMFLY polynomials, respectively. 
 We split the proof  in two steps.
%********************************************************************
\begin{lemma}\label{lm3.5}
The order on the braid group is Dehornoy if and only if the  positive 
braids correspond  to the positive Laurent polynomials. 
\end{lemma}
%*********************************************************************
\begin{proof}
(i) It follows from [Franks \& Williams 1987] \cite[Theorem 2.2 (1)]{FraWil1} 
and [Ito 2020] \cite{Ito1} that the ordering of $B_{2g+n}$ is Dehornoy
if and only if  the skein relations applied to the closure $\hat{}$ of braids 
$x,y,z\in B_{2g+n}$ preserve the semigroup $B_{2g+n}^+$. 
(In other words, if $x,  y\in B_{2g+n}^+$, then $z\in B_{2g+n}^+$,
where the Laurent polynomials of $\hat{x}, \hat{y}, \hat{z}$ are connected
by the skein relation.)  

\medskip
(ii) On the other hand, it is known  \cite[Section 4.4.6]{N} that the skein relations transform 
to the exchange relations in the algebra $\mathcal{A}(\mathbf{x}, S_{g,n})$. 
Therefore, the ordering on $B_{2g+n}$ is Dehornoy if and only if the exchange
relations preserve the semigroup of the Laurent polynomials with positive coefficients. 

\medskip
(iii)  But the exchange relations always preserve the positive semigroup of the 
cluster algebra  $\mathcal{A}(\mathbf{x}, S_{g,n})$.  The latter is known as 
the Positivity Conjecture proved for all surface cluster algebras [Williams 2014] \cite[Conjecture 2.28]{Wil1}. 

\bigskip
Lemma \ref{lm3.5} follows from the items (i)-(iii). 
\end{proof}
%*********************************************************************

%********************************************************************
\begin{corollary}\label{cor3.6}
The braid ordering (i) in theorem \ref{thm1.1} 
is  Dehornoy's ordering.  
\end{corollary}
%*********************************************************************
\begin{proof}
In view of remark \ref{rmk2.7} and commutative diagram in Figure 1, 
we conclude  that  the  positive 
braids correspond  to the positive Laurent polynomials. 
It remains to apply the conclusion of lemma \ref{lm3.5}.
Therefore the  braid ordering (i) in theorem \ref{thm1.1} 
is isomorphic to the  Dehornoy  ordering. 
Corollary \ref{cor3.6} is proved. 
\end{proof}
%*********************************************************************

\bigskip
%**************************************************************************
\subsection{Proof of corollary \ref{cor1.2}}
%***************************************************************************
For the sake of clarity, let us outline the main ideas.  
Consider a $C^*$-algebra $\mathbb{A}_{\Theta}:=\mathbb{A}(S_{g,n})/I_{\Theta}$,
where $I_{\Theta}\subset \mathbb{A}(S_{g,n})$ is the primitive ideal and 
$\Theta=(1,\theta_1\dots,\theta_{6g-7+2n})\in \mathbb{P} T_{g,n}$
\cite[Theorem 2]{Nik1}. 
(The ideal $I_{\Theta}$ is given a subgraph of the Bratteli diagram 
of the AF-algebra $\mathbb{A}(S_{g,n})$ obtained by an excision of the diagram along an
infinite path convergent  to $\Theta$, see \cite{Nik1} for the
details and examles.)
The  $\mathbb{A}_{\Theta}$ is a simple AF $C^*$-algebra
and $K_0(\mathbb{A}_{\Theta})\cong \mathbf{Z}+\mathbf{Z}\theta_1+\dots
+\mathbf{Z}\theta_{6g-7+2n}$. 
We prove that $C_r^*\left(\widehat{\pi_1(S_{g,n})}\right)\rtimes_{\alpha_{\Theta}}\mathbf{Z}\subset \mathbb{A}_{\Theta}$,
where $\alpha_{\Theta}\in Aut~\left(C_r^*\left(\widehat{\pi_1(S_{g,n})}\right)\right)$ and $\widehat{\pi_1(S_{g,n})}$
the profinite completion of the fundamental group $\pi_1(S_{g,n})$. The latter
inclusion gives rise to a positive semigroup   $\pi_1^+(S_{g,n})$
depending on $\Theta\in\mathbb{P} T_{g,n}$. Thus one gets a space
of different left-invariant ordering of  $\pi_1(S_{g,n})$ dual to 
the projective Teichm\"uller space of the surface $S_{g,n}$. 
We pass to a detailed argument by 
splitting  the proof in a series of lemmas.

\bigskip
%********************************************************************
\begin{lemma}\label{lm3.7}
For each $\Theta\in \mathbb{P} T_{g,n}$ there exists an automorphism
$\alpha_{\Theta}$ of $C_r^*\left(\widehat{\pi_1(S_{g,n})}\right)$ such that
%****************************************************************************
\begin{equation}\label{eq3.9}
C_r^*\left(\widehat{\pi_1(S_{g,n})}\right)\rtimes_{\alpha_{\Theta}}\mathbf{Z}\subset \mathbb{A}_{\Theta}.
\end{equation}
%******************************************************************************
\end{lemma}
%*********************************************************************
\begin{proof}
(i)
Let $b\in B_{2g+n}$ be a braid.  Denote by $\mathscr{L}_b$  a link obtained by the closure of $b$
 and let  $\pi_1(\mathscr{L}_b)$ be the  fundamental group of $\mathscr{L}_b$.
Recall that  
%*****************************************************************************
\begin{equation}\label{eq3.10}
\pi_1(\mathscr{L}_b)\cong\langle x_1,\dots, x_{2g+n} ~|~ x_i=r(b)x_i, ~1\le i\le 2g+n\rangle,
\end{equation}
%*************************************************************************************************
where $x_i$ are generators of the free group $F_{2g+n}$ and  $r: B_{2g+n}\to Aut~(F_{2g+n})$ is 
the Artin representation of $B_{2g+n}$.  Let  $\mathcal{I}_b$ be  a two-sided   ideal  in the algebra  $\mathbb{A}(S_{g,n})$
 generated by  relations (\ref{eq3.10}).   
It is easy to see that the ideal  $\mathcal{I}_b$ is self-adjoint and thus (\ref{eq1.1})  induces  a  representation 
 %*********************************************************************************
\begin{equation}\label{eq3.11}
R: ~\pi_1(\mathscr{L}_b)\to \mathbb{A}(S_{g,n})~/~\mathcal{I}_b :=\mathbb{A}_b.
\end{equation}
%*******************************************************************************
Notice that the $\mathbb{A}_b$ is an AF-algebra of stationary type \cite[Section 3.5.2]{N} given 
by a positive matrix $A\in GL_{6g-6+2n}(\mathbf{Z})$. 
Therefore it is stably isomorphic to an AF-algebra  $\mathbb{A}_{\Theta}$, where 
$\Theta=(1,\theta_1\dots,\theta_{6g-7+2n})$
is an eigenvector corresponding to
the Perron-Frobenius eigenvalue of  the matrix $A$.

\medskip
(ii)  Since $\pi_1(S_{g,n})\cong F_{2g+n-1}$, one can write (\ref{eq3.10}) in the form:
%*********************************************************************************
\begin{equation}\label{eq3.12}
\pi_1(\mathscr{L}_b)\cong \pi_1(S_{g,n})\rtimes_{\alpha}\mathbf{Z},
\end{equation}
%*******************************************************************************
where the crossed product is taken by an automorphism $\alpha\in Aut~(\pi_1(S_{g,n}))$. 
It follows from (\ref{eq3.11}) that for an $\alpha_{\Theta}\in Aut~(C_r^*(\pi_1(S_{g,n}))$
one gets an inclusion of the $C^*$-algebras:
 %*********************************************************************************
\begin{equation}\label{eq3.13}
C_r^*(\pi_1(S_{g,n}))\rtimes_{\alpha_{\Theta}}\mathbf{Z}\subset \mathbb{A}_{\Theta}. 
\end{equation}
%*******************************************************************************

\medskip
(iii) It remains to extend the inclusion (\ref{eq3.13}) to the profinite completion
of the fundamental group $\pi_1(S_{g,n})$ as follows. Since $\pi_1(S_{g,n})$
is a residually finite group, we obtain a dense embedding of the groups:
 %*********************************************************************************
\begin{equation}\label{eq3.14}
\pi_1(S_{g,n})\hookrightarrow \widehat{\pi_1(S_{g,n})}. 
\end{equation}
%*******************************************************************************

Passing to the reduced group $C^*$-algebras, we get a dense embedding 
of the $C^*$-algebras: 
  %*********************************************************************************
\begin{equation}\label{eq3.15}
C^*_r(\pi_1(S_{g,n}))\hookrightarrow C^*_r\left(\widehat{\pi_1(S_{g,n})}\right). 
\end{equation}
%*******************************************************************************

The crossed product 
%****************************************************************************
\begin{equation}\label{eq3.16}
C_r^*\left(\widehat{\pi_1(S_{g,n})}\right)\rtimes_{\alpha_{\Theta}}\mathbf{Z}\subset \mathbb{A}_{\Theta}
\end{equation}
%******************************************************************************
can be defined as the norm-closure of (\ref{eq3.13}), where $\alpha_{\Theta_i}$ 
converge to an automorphism   $\alpha_{\Theta}\in Aut~\left(C_r^*\left(\widehat{\pi_1(S_{g,n})}\right)\right)$.
Clearly, the inclusion (\ref{eq3.16}) is strict because  $\mathbb{A}_{\Theta}$ is an AF $C^*$-algebra
with the trivial $K_1$-group unlike the embedded crossed product $C^*$-algebra. 

\bigskip
This argument finishes the proof of lemma \ref{lm3.7}. 
\end{proof}
%*********************************************************************

\bigskip
%********************************************************************
\begin{lemma}\label{lm3.8}
$K_0\left(C_r^*\left(\widehat{\pi_1(S_{g,n})}\right)\rtimes_{\alpha_{\Theta}}\mathbf{Z}\right)\cong K_0(\mathbb{A}_{\Theta}).$
\end{lemma}
%*********************************************************************
\begin{proof}
Roughly speaking, the proof follows from lemma \ref{lm3.1} (ii).
Namely, an isomorphism of the dimension groups
%********************************************************
\begin{equation}\label{eq3.17}
\left(K_0(A_{2g+n}),  ~K_0^+(A_{2g+n})\right)\cong \left(K_0(\mathbb{A}(S_{g,n})),  ~K_0^+(\mathbb{A}(S_{g,n}))\right)
\end{equation}
%***********************************************************************
implies that for an order-ideal $K_0(\mathcal{I}_b)$ of  $K_0^+(\mathbb{A}(S_{g,n}))$,
one gets from (\ref{eq3.11}) an isomorphism:
%********************************************************
\begin{equation}\label{eq3.18}
K_0\left(C_r^*\left(\pi_1(S_{g,n})\right)\rtimes_{\alpha_{\Theta}}\mathbf{Z}\right)\cong K_0(\mathbb{A}_{\Theta}).
\end{equation}
%***********************************************************************
Passing to the profinite completion at the LHS of (\ref{eq3.18}),  we obtain the conclusion of
lemma \ref{lm3.8}. 
\end{proof}
%*********************************************************************

%**************************************************************************
%*******************************************************************
\begin{figure}
%*******************************************************************
\begin{picture}(300,90)(-40,20)
\put(120,70){\vector(0,-1){30}}
%\put(130,70){\vector(0,-1){35}}
\put(30,70){\vector(1,-1){30}}
\put(60,83){\vector(1,0){20}}
\put(65,25){$\mathbf{Z}+\mathbf{Z}\theta_1 +\dots+\mathbf{Z}\theta_{6g-7+2n}$}
\put(17,80){$\pi_1(S_{g,n})$}
\put(85,80){$C_r^*\left(\widehat{\pi_1(S_{g,n})}\right)\rtimes_{\alpha_{\Theta}}\mathbf{Z}$}
\put(70,90){$\iota$}
\put(125,55){$K_0$}
\end{picture}
%***********************************************************
\caption{}
\end{figure}
%*******************************************************************

%********************************************************************
\begin{lemma}\label{lm3.9}
Let $\pi_1^+(S_{g,n})$ be a semigroup obtained as a pull back of the positive
cone $\mathbf{Z}+\mathbf{Z}\theta_1 +\dots+\mathbf{Z}\theta_{6g-7+2n}>0$ on the
commutative diagram in Figure 2, where $\iota$ is an embedding induced by (\ref{eq3.14}). 
Then  each $(1,\theta_1,\dots,\theta_{6g-7+2n})\in \mathbb{P}T_{g,n}$ 
defines a left-invariant order on the fundamental group $\pi_1(S_{g,n})$
if and only if $\theta_i\in\mathbf{R}$ are rationally independent. 
\end{lemma}
%*********************************************************************
\begin{proof}
(i) Let us show that the semigroup $\pi_1^+(S_{g,n})$ depends on $6g-7+2n$ 
variables $\theta_i$.  Indeed, lemma \ref{lm3.7} says that automorphisms 
$\alpha_{\Theta}$ of $C_r^*\left(\widehat{\pi_1(S_{g,n})}\right)$ are parameterized 
by $\Theta=(1,\theta_1,\dots,\theta_{6g-7+2n})\in \mathbb{P}T_{g,n}$. 
From lemma \ref{lm3.8} and remark \ref{rmk2.7} we conclude that 
the semigroup $\pi_1^+(S_{g,n})$ obtained from the diagram in Figure 2,
will vary as $(1,\theta_1,\dots,\theta_{6g-7+2n})$ runs all rationally independent values of $\theta_i$. 
%********************************************************************
\begin{remark}
It is easy to see that if $\theta_i$ are rationally dependent,  then  the rank of the algebra $\mathbb{A}_b$
in (\ref{eq3.11})  is less than $6g-6+2n$ and thus $R$ is no longer  a faithful representation.   
Therefore the condition of rational independence for the $\theta_i$ cannot be omitted.  
\end{remark}
%*********************************************************************

\bigskip
(ii) Let us show that the order of $\pi_1(S_{g,n})$ defined by the
semigroup $\pi_1^+(S_{g,n})$ is invariant of the left-multiplication. 
Roughly speaking, the proof repeats all the steps in the proof of 
lemma \ref{lm3.4}.  Namely, one takes the canonical tracial state 
$\tau$ on the $C^*$-algebra  $C_r^*\left(\widehat{\pi_1(S_{g,n})}\right)\rtimes_{\alpha_{\Theta}}\mathbf{Z}$
and  writes  equations (\ref{eq3.3})--(\ref{eq3.8}) for $\tau$. The details are left to the reader.

\bigskip
Lemma \ref{lm3.9} is proved. 
\end{proof}
%*********************************************************************

%********************************************************************
\begin{corollary}\label{cor3.10}
$LO (\pi_1(S_{g,n}))\cong \mathscr{K}\subset\mathbb{P} T_{g,n}.$
\end{corollary}
%*********************************************************************
\begin{proof}
Using lemma \ref{lm3.9}, we identify the space $LO (\pi_1(S_{g,n}))$ of all left-invariant orderings 
of the fundamental group $\pi_1(S_{g,n})$ with the a totally disconnected subset $\mathscr{K}\subset\mathbb{P} T_{g,n}$
consisting of  the vectors $(1,\theta_1,\dots,\theta_{6g-7+2n})$  with the rationally independent components $\theta_i$. 
\end{proof}
%*********************************************************************

\bigskip
%********************************************************************
\begin{corollary}\label{cor3.11}
Each left-ordering of $\pi_1(S_{g,n})$ defines a Riemann
surface $S_{g,n}$ up to the action of the Teichm\"uller geodesic flow on 
$T_{g,n}$.  
\end{corollary}
%*********************************************************************
\begin{proof}
Recall that the Teichm\"uller space $T_{g,n}$ is a fiber bundle over $\mathbb{P}T_{g,n}$ 
with the fiber $\mathbf{R}$ \cite[Theorem 2]{Nik1}. Moreover,   
each fiber $\mathbf{R}$ is an orbit  of the 
Teichm\"uller geodesic flow $\{T^t: T_{g,n}\to T_{g,n}~|~t\in\mathbf{R}\}$ 
\cite[Section 4]{Nik1}.  

On the other hand, corollary \ref{cor3.10} says that 
each left-ordering of $\pi_1(S_{g,n})$ defines a fiber $\mathbf{R}$ in $T_{g,n}$.
The fiber consists of all Riemann surfaces $S_{g,n}$ which lie at the same orbit
of the geodesic flow $T^t$. Corollary \ref{cor3.11} follows.
\end{proof}
%*********************************************************************

\bigskip
Corollaries \ref{cor3.10} and \ref{cor3.11} finish the  proof of corollary \ref{cor1.2}. 

%**************************************************************************
%\subsection{Proof of remark \ref{rmk1.4}}
%***************************************************************************

%**************************************************************************
\section{The ordering space of free groups}
%***************************************************************************
The free groups $F_m$  are  orderable  [Rolfsen 2014] \cite[Proposition 2.9]{Rol1}. 
The space $LO(F_m)$ can be evaluated using corollary
\ref{cor1.2} combined with the well-known formula $\pi_1(S_{g,n})\cong F_{2g+n-1}$.
%Namely, the following is true. 
%********************************************************************
\begin{corollary}\label{cor4.1}
%*******************************************************************
\begin{equation}\label{eq4.1}
LO(F_m) \cong 
\begin{cases}
\mathscr{K}\subset\mathbf{S}^{6k-5},&\hbox{if} \quad m=2k\cr
 \mathscr{K}\subset\mathbf{S}^{6k-3}, & \hbox{if} \quad m=2k+1.
\end{cases}
\end{equation}
%*********************************************************************
\end{corollary}
%*********************************************************************
\begin{proof}
(i) Since the number of cusps $n\ge 1$,  we conclude that  $\pi_1(S_{g,n})\cong F_{2g+n-1}$
is a free group. Therefore  corollary \ref{cor1.2} says that:
%**************************************************************************************
\begin{equation}\label{eq4.2}
LO ( F_{2g+n-1})\cong \mathscr{K}\subset \mathbb{P}T_{g,n}\cong \mathbf{S}^{6g-7+2n}.
\end{equation}
%***********************************************************************************

\medskip
(ii) On the other hand, from the representation (\ref{eq2.3}) we have a restriction  $n\in\{1,2\}$. 
Letting  $g=k$ in (\ref{eq4.2}), one gets the formula (\ref{eq4.1}) for  $n=1$ and  $n=2$. 
\end{proof}
%*****************************************************************

%*****************************************************************
\begin{example}
We conclude by an example of the free group $F_2$ of rank 2. 
Such a group is the fundamental group of the once-punctured torus
$S_{1,1}$ \cite[Example 1]{Nik1}. In other words, we have $g=n=k=1$.  
From corollary \ref{cor4.1}, one gets:
%**************************************************************************************
\begin{equation}\label{eq4.3}
LO ( F_2)\cong \mathscr{K}\subset\mathbf{S}^1,
\end{equation}
%***********************************************************************************
where $\mathscr{K}$ is the set of all irrational 
numbers of the circle $\mathbf{S}^1=\mathbf{R}/\mathbf{Z}$.  On the other hand, by corollary \ref{cor1.2} each irrational $\theta\in\mathscr{K}$ 
defines a Riemann surface $S_{1,1}$ modulo the action of the Teichm\"uller flow $T^t$. 
This fact  can be proved independently using the notion of 
a non-commutative torus $\mathcal{A}_{\theta}$ \cite[Section 1.3]{N}.
\end{example}
%******************************************************************

\section*{Conflict of interest}  The authors declare that they have no conflict of interest.

\bigskip\noindent
{\sf Acknowledgement.} 
  I thank the referee for helpful comments.

\bibliographystyle{amsplain}

%**********************************************************

\end{document}